\documentclass[10pt]{amsart} 
\usepackage{amsmath,amssymb}
\usepackage[dvipdfmx]{graphicx}
\usepackage{float}
\usepackage{multicol}
\usepackage{ascmac}
\usepackage{mathtools}
\usepackage{bm}
\usepackage{amscd}
\usepackage{comment}
\usepackage{mathrsfs}
\usepackage[all]{xy}
\title[weighted sum formula for FAMZV]{Weighted sum formulas for finite alternating multiple zeta values with some parameters
}
\author{Takumi Anzawa}
\email{m20001s@math.nagoya-u.ac.jp}
\address{graduate school of mathematics, nagoya university, furo-cho, chikusa-ku, nagoya, 464-8602, japan}
\date{2022/01/06} 
\usepackage[margin=35mm]{geometry}
\newtheorem{df}{Definition}[section]%[subsection]

\newtheorem{lem}[df]{Lemma}
\newtheorem{cor}[df]{Corollary}
{\theoremstyle{remark}
\newtheorem{rmk}[df]{Remark}}
\newtheorem{thm}[df]{Theorem}

\newtheorem*{mt}{Main\;\;Theorem}

\DeclareMathOperator{\dep}{dp}
\DeclareMathOperator{\sgn}{sgn}
\DeclareMathOperator{\zf}{\zeta_{\mathscr{A}}}
\DeclareMathOperator{\wt}{wt}
\DeclareMathOperator{\Li}{Li}

\numberwithin{equation}{section}

\allowdisplaybreaks[1]
\begin{document}
\maketitle
\begin{abstract}
We prove a sum formula with 4 parameters among finite alternating multiple zeta values which can be regarded as an alternating version of result of Kamano on finite multiple zeta values.
\end{abstract}
\tableofcontents
\section{Introduction}
%In the present paper, $\mathbb{N}$ is set of positive integers. 
The multiple zeta value (MZV, in short) is the real number defined by
\[
\zeta(\mathbf{k}):=\zeta(k_1,\ldots,k_r):=\sum_{0<m_1<\cdots<m_r}\frac{1}{m_1^{k_1}\cdots m_r^{k_r}}\in\mathbb{R}.
\]
for $r\in\mathbb{N}$ and an index $\mathbf{k}:=(k_1,\ldots,k_r)\in\mathbb{N}^r$ with $k_r\ge 2$.
The condition  $k_r\ge 2$ ensures the convergence of the series.
This number has deep properties and has appeared in recent years
in connection with a surprising diversity of topics, including knot invariants (cf. \cite{LM}), periods of mixed Tate motives (cf. \cite{T}), and calculations of integrals
associated with Feynman diagrams in perturbative quantum field theory (cf. \cite{B2}).

Kaneko and Zagier introduced a finite analogue of multiple zeta value which belongs to the following $\mathbb{Q}$-algebra 
\[
\mathscr{A}:=\left(\prod_{p:\text{prime}}\mathbb{Z}/p\mathbb{Z}\right)\Bigm/\left(\bigoplus_{p :\text{prime}} \mathbb{Z}/p\mathbb{Z}\right).
\]  
For each index $(k_1,\ldots,k_r)\in\mathbb{N}^r$, the finite multiple zeta values (FMZV, in short) is defined by 
\[
\zf(k_1,\ldots,k_r):=\left(\sum_{0<m_1<\cdots<m_r<p} \frac{1}{m_1^{k_1}\cdots m_r^{k_r}}\mathrm{\mod}\;p\right)_p\in\mathscr{A}.
\]
Several relations among FMZV's are found in \cite{IKT}, \cite{SW} etc.
%FMZVs have some relation (cf. \cite[THEOREM 1.4]{SW}). 
%\section{Finite Alternating Multiple Zeta Values}
We define $\overline{\mathbb{N}}:=-\mathbb{N}$ and $\overline{n}$ as element of $\overline{\mathbb{N}}$ denotes $-n$ for $n\in\mathbb{N}$ and we define $\overline{\overline{n}}=n$.
Let $\mathbb{D}$ be the disjoint union $\mathbb{N}\cup\overline{\mathbb{N}}$ that is, $\mathbb{D}=\mathbb{Z}\setminus\{0\}$. The signature $\sgn$ and the absolute value on $\mathbb{D}$ are defined by
\[
\sgn (n):=
\left\{
\begin{array}{cc}
1&(n\in\mathbb{N})\\
-1&(n\in\overline{\mathbb{N}})
\end{array},\;\;\;
\right. 
|n|:=
\left\{
\begin{array}{cc}
n&(n\in\mathbb{N})\\
\overline{n}&(n\in\overline{\mathbb{N}})
\end{array}
\right. .
\]
We call an element of $\mathbb{D}^r$ also an index. 
For each index $\bm{\alpha}=(\alpha_i)_{1\le i\le r}\in\mathbb{D}^r$,
its weight, denoted by $\wt(\bm{\alpha})$ (or $|\bm{\bm{\alpha}}|$ for simplicity), means $|\alpha_1|+\cdots+|\alpha_r|$ and its depth, denoted by $\dep (\bm{\alpha})$, means $r$.
For such an index $\bm{\alpha}$, the corresponding finite alternating multiple zeta value (FAMZV, in short) is defined by
\[
\zf(\alpha_1,\ldots,\alpha_r):=\left(\sum_{0<m_1<\cdots<m_r<p} \frac{\sgn(\alpha_1)^{m_1}\cdots \sgn(\alpha_r)^{m_r}}{m_1^{|\alpha_1|}\cdots m_r^{|\alpha_r|}}\mathrm{\mod}\;p\right)\in\mathscr{A}
\]
(cf. \cite{zhao}). % In this paper, $m^{|k|}$ ($m\in\mathbb{Z}$, $k\in\mathbb{D}$) writes $m^k$ for simplicity.
 The following formula is proved in \cite{Ka}.

\begin{thm}[\cite{Ka}, Main Theorem]\label{Kama}
Let $\lambda_1$, $\lambda_2$, $\mu_1$ and $\mu_2$ be indeterminates. For any non-negative integers $q_1$ and $q_2$, the following holds in $\mathscr{A}[\lambda_1,\lambda_2,\mu_1,\mu_2]$.
\begin{align}
&\sum_{\substack{i_1+i_2=q_1\\ j_1+j_2=q_2}}\left\{(-1)^{i_2+j_2}\lambda_1^{i_1}\lambda_2^{i_2}\mu_1^{j_1}\mu_2^{j_2}+(\lambda_1^{i_1}\mu_1^{j_1}+\lambda_2^{i_2}\mu_2^{j_2})(\lambda_1+\lambda_2)^{i_2}(\mu_1+\mu_2)^{j_2}\right\}\sum_{\substack{\bm{\alpha}\in S_{i_1,j_1}\\ \bm{\beta}\in S_{i_2,j_2}}}\zf(\bm{\alpha},\bm{\beta})=(0)_p . \label{kamano}
\end{align}
where $S_{i,j}$ is defined by $S_{i,j}:=\{\bm{\alpha}\in\mathbb{N}^{i+1}\mid |\bm{\alpha}|=i+j+1\}$ for $i$, $j\in\mathbb{N}$. 
\end{thm}
The main result of this paper is an alternating analog of Theorem \ref{Kama}:
\begin{mt}\label{an}
Let $\lambda_1$, $\lambda_2$, $\mu_1$ and $\mu_2$ be indeterminates. For any non-negative integers $n$, m,  the following holds in $\mathscr{A}[\lambda_1,\lambda_2,\mu_1,\mu_2]$
\begin{align}
&\sum_{\substack{r_1+r_2=n\\ k_1+k_2=m}}(-1)^{r_2+k_2}\lambda_1^{r_1}\lambda_2^{r_2}\mu_1^{k_1}\mu_2^{k_2} 
\sum_{\substack{\bm{\alpha}\in S_{r_1,k_1}^{\sgn}\\ \bm{\beta}\in S_{r_2,k_2}^{\sgn}}}\sgn(\bm{\beta}) \zf(\bm{\alpha},\bm{\beta}) \label{AA}\\
&+\sum_{\substack{r_1+r_2=n \\ k_1+k_2=m}}(\lambda_1^{r_1}\mu_1^{k_1}+\lambda_2^{r_1}\mu_2^{k_1})(\lambda_1+\lambda_2)^{r_2}(\mu_1+\mu_2)^{k_2} \sum_{\substack{\bm{\alpha}\in S_{r_1,k_1}^{\sgn}\\ \bm{\beta}\in S_{r_2,k_2}^{\sgn}}}\zf(\bm{\alpha},\bm{\beta})=(0)_p. \notag
\end{align}
where $S_{i,j}^{\sgn}$ is defined by  $S_{i,j}^{\sgn}:=\{\bm{\bm{\alpha}}\in\mathbb{D}^{i+1}\mid |\bm{\bm{\alpha}}|=i+j+1\}$ for $i$, $j$, $\in\mathbb{Z}_{\ge 0}$.
\end{mt}

\section{Proof of main theorem}
For each index $\bm{\alpha}=(\alpha_1,\ldots,\alpha_r)\in\mathbb{D}^r$, the multiple polylogarithm is the complex function defined by the following series
\[
\Li(\bm{\alpha};z):=\Li(\bm{\alpha}):=\sum_{0<m_1<\cdots<m_r} \frac{\sgn(\alpha_1)^{m_1}\cdots \sgn(\alpha_r)^{m_r}}{m_1^{\alpha_1}\cdots m_r^{\alpha_r}}z^{m_r}.
\]
The radius of convergence is equal to $1$ for each $\bm{\alpha}\in \mathbb{D}^r$.
We note that the above function has an analytic continuation to a domain, bigger than the open unit disk, by the integral expression called the iterated integral expression. 

Let $\iota_i:=\sgn(\alpha_i)$ and $\eta_i:=\prod_{j=1}^i\iota_{r+1-j}$.
For $0<z<1$, the iterated integral expression of the multiple polylogarithm is equal to
\[
\Li(\bm{\alpha},z)=\int \cdots \int_{0<t_1<\cdots<t_{|\bm{\alpha}|}<z}
\frac{dt_1}{\eta_r-t_1}\underbrace{\frac{dt_2}{t_2}\cdots \frac{dt_{|\alpha_1|}}{t_{|\alpha_1|}}}_{|\alpha_1|-1}\cdots\frac{dt_{|\bm{\alpha}|-|\alpha_r|+1}}{\eta_1-t_{|\bm{\alpha}|-|\alpha_r|+1}}\underbrace{\frac{dt_{|\bm{\alpha}|-|\alpha_r|+2}}{t_{|\bm{\alpha}|-|\alpha_r|+2}}\cdots \frac{dt_{|\bm{\alpha}|}}{t_{|\bm{\alpha}|}}}_{|\alpha_r|-1}.
\]
For any primes $p$,  we define $\mathbb{Q}$-linear operators $\mathfrak{L}_p:\mathbb{Q}[[z]]\rightarrow \mathbb{Q}$ by 
\[
\mathfrak{L}_p\left(\sum_{n=0}^\infty a_n z^n\right)=\sum_{n=0}^{p-1} a_n.
\]
The function $\Li(\bm{\alpha};z)$ regarded as a formal power series over $\mathbb{Q}$
 satisfies $(\mathfrak{L}_p (\Li(\bm{\alpha})) \mathrm{\mod}\;p)_p=\zf(\bm{\alpha})$.

For any $t_1$, $t_2\in(0,1)$, we define  
\begin{align*}
L_1(t_1,t_2):=&\int_{t_1}^{t_2}\frac{dt}{1-t}=\log \frac{1-t_1}{1-t_2}\\
L_{-1}(t_1,t_2):=&\int_{t_1}^{t_2}\frac{-dt}{1+t}=\log \frac{1+t_1}{1+t_2}\\
L_0(t_1,t_2):=&\int_{t_1}^{t_2}\frac{dt}{t}=\log \frac{t_2}{t_1}.
\end{align*}
The `chain rule' holds for these functions.
\begin{align}
L_j(t_1,t_2)=L_j(t_1,u)+L_j(u,t_2)\;\;(j\in\{\pm1,0\}). \label{ch}
\end{align}
The following is an alternating generalization of \cite{EW} Proposition 2.1 
%The idea of following lemma was shown in \cite{EW} Proposition 2.1.
\begin{lem}\label{lem1}
For $s\in\mathbb{N}$ and $q_l$ and $k_l\in\mathbb{N}$ ($1\le l \le s$) and $0<z<1$, we have
\begin{align*}
&\sum_{\substack{\eta_l\in\{\pm1\}\\ 1\le l \le s}}\sum_{i_1+j_1=q_1}\cdots \sum_{i_s+j_s=q_s}\frac{1}{i_1!\cdots i_s! j_1!\cdots j_s!k_1!\cdots k_s!}\int_{0<t_1<\cdots<t_s<z} L_1^{i_1}(t_1,t_2)\cdots L_1^{i_s}(t_s,z)\\
&\times L_{-1}^{j_1}(t_1,t_2)\cdots L_{-1}^{j_s}(t_s,z)L_0^{k_1}(t_1,t_2)\cdots L_0^{k_s}(t_s,z)\frac{dt_1}{\eta_1-t_1}\cdots \frac{dt_s}{\eta_s-t_s} \\
&= \sum_{(\bm{\alpha}_1,\ldots,\bm{\alpha}_s)\in S_s} \Li(\bm{\alpha}_1,\ldots,\bm{\alpha}_s;z)
\end{align*}
where $S_s$ is defined by $\displaystyle \prod_{1\le l \le s} S_{q_l,k_l}^{\sgn}$.
\end{lem}
\begin{proof}
For its proof, see Appendix \ref{Ap}. 
\begin{comment}
We prove only $s=1$ case of the above formula.
if $\eta_1$ is equal to $1$ (resp. $-1$), the first component of index is in $\mathbb{N}$. (resp. $\mathbb{D}$)
Hence, we note it appear different indices for $\eta_1=1$ or $\eta_1=-1$. 
For each $\bm{k}\in\mathbb{D}^{i_1+j_1+1}$ with $\wt \bm{k}=i_1+j_1+k_1+1$, it has the unique iterated integral expression. Hence, the permutation of $L_0(t_1,t_2)^{i_1}$, $L_{1}(t_1,t_2)^{j_1}$ and $L_{-1}(t_1,t_2)^{k_1}$ is $i_1!$, $j_1!$ and $k_1!$, it yields the conclusion.
\end{comment}
\end{proof}

\begin{lem}\label{Lem2}
Let $s$, $t\in\mathbb{N}$ and take indices $\bm{\alpha}:=(\alpha_1,\ldots,\alpha_s)\in\mathbb{D}^s$, $\bm{\beta}:=(\beta_1,\ldots,\beta_t)\in\mathbb{D}^t$. We have
\[
\left(\mathfrak{L}_p(\Li(\bm{\alpha})\Li(\bm{\beta}))\mathrm{\mod}\;p \right)_p
=\sgn(\bm{\beta})(-1)^{\wt \bm{\beta}}\zf(\alpha_1,\ldots,\alpha_s,\beta_t,\ldots,\beta_1).
\]
in $\mathscr{A}$ with $\sgn(\bm{\beta}):=\prod_{i=1}^t \sgn(\beta_i)$.
\end{lem}
\begin{proof}
If $s+t<p$,  we have
\begin{align*}
\mathfrak{L}_p(\Li(\bm{\alpha})\Li(\bm{\beta}))=&
\sum_{\substack{0<m_1<\cdots<m_s<p\\ 0<n_1<\cdots<n_t<p\\ m_s+n_t<p}}
\frac{\sgn(\alpha_1)^{m_1}\cdots \sgn(\alpha_s)^{m_s}\sgn(\beta_1)^{n_1}\cdots \sgn(\beta_t)^{n_t}}{m_1^{\alpha_1}\cdots m_s^{\alpha_s}n_1^{\beta_1}\cdots n_t^{\beta_t}}.
\intertext{Let $n_i':=p-n_i$ for $1\le i\le t$, by $m_s+n_t<p$, we have the order $0<m_1<\cdots<m_s<p-n_{t}<\cdots <p-n_1<p$, it means $0<m_1<\cdots<m_s<n_{t}'<\cdots <n_1'<p$. We obtain}
=&\sum_{0<m_1<\cdots<m_s<n_{t}'<\cdots <n_1'<p}
\frac{\sgn(\alpha_1)^{m_1}\cdots \sgn(\alpha_s)^{m_s}\sgn(\beta_1)^{p-n_1'}\cdots \sgn(\beta_t)^{p-n_t'}}{m_1^{\alpha_1}\cdots m_s^{\alpha_s}(p-n_t')^{\beta_t}\cdots (p-n_1')^{\beta_a}}\\
\intertext{Since we have $\sgn (\beta_i)^k=\sgn (\beta_i)^{-k}$ for each $1\le i\le t$ and $\sgn (\beta_i)^p=\sgn(\beta_i)$ for each odd prime $p$, we have}
\equiv&
\begin{multlined}[t]
\sum_{0<m_1<\cdots<m_s<n_{t}'<\cdots <n_1'<p} (-1)^{|\bm{\beta}|}\sgn(\bm{\beta})\\
\quad\frac{\sgn(\alpha_1)^{m_1}\cdots \sgn(\alpha_s)^{m_s}\sgn(\beta_t)^{n_t'}\cdots \sgn(\beta_1)^{n_1}}{m_1^{\alpha_1}\cdots m_s^{\alpha_s}n_t'^{\beta_t}\cdots n_1'^{\beta_1}}\;\;\;(\mathrm{mod}\;p).
\end{multlined}
\end{align*}
Hence for almost all primes $p$, the above formula holds and we obtain the claim.
\end{proof}

\textit{Proof of Main Theorem.}

Let $0<z<1$, we define
\begin{align*}
N^{\sgn}(z):=&\sum_{\eta_1,\eta_2\in\{\pm 1\}}\sum_{q_1+q_2=n}\frac{1}{q_1!q_2!m!}\int_{\substack{0<s<z\\ 0<t<z}}
(\lambda_1L_1(s,z)+\lambda_2L_1(t,z))^{q_1}
(\lambda_1 L_{-1}(s,z)+\lambda_2 L_{-1}(t,z))^{q_2}\\
&(\mu_1 L_{0}(s,z)+\mu_2 L_{0}(t,z))^{m}
\frac{ds}{\eta_1-s}\frac{dt}{\eta_2-t}.
\end{align*}
By the binomial theorem, we have
\begin{align*}
N^{\sgn}(z)=&\sum_{\eta_1,\eta_2\in\{\pm 1\}}\sum_{q_1+q_2=n} \sum_{\substack{i_1+i_2=q_1\\ j_1+j_2=q_2\\ k_1+k_2=m}} \frac{1}{i_1!i_2!}\frac{1}{j_1!j_2!}\frac{1}{k_1!k_2!}\\
&\times\int_{0<s<z} \lambda_1^{i_1+j_1}\mu_1^{k_1}L_1(s,z)^{i_1}L_{-1}(s,z)^{j_1}L_0(s,z)^{k_1}\frac{ds}{\eta_1-s}\\
&\times\int_{0<t<z} \lambda_2^{i_2+j_2}\mu_2^{k_2}L_1(t,z)^{i_2}L_{-1}(t,z)^{j_2}L_0(t,z)^{k_2}\frac{dt}{\eta_2-t}\\
%%%%%%%%%%%%%%%%%%%%%%
=& \sum_{\substack{q_1+q_2=n\\ i_1+i_2=q_1\\ j_1+j_2=q_2\\ k_1+k_2=m}}\lambda_1^{i_1+j_1}\lambda_2^{i_2+j_2}\mu_1^{k_1}\mu_2^{k_2} \\
&\times\left(\frac{1}{i_1!j_1!k_1!}\int_{0<s<z}L_1(s,z)^{i_1}L_{-1}(s,z)^{j_1}L_0(s,z)^{k_1}\left(\frac{ds}{1-s}+\frac{-ds}{1+s}\right)\right)\\
&\times\left(\frac{1}{i_2!j_2!k_2!}\int_{0<t<z}L_1(t,z)^{i_2}L_{-1}(t,z)^{j_2}L_0(t,z)^{k_2}\left(\frac{dt}{1-t}+\frac{-dt}{1+t}\right)\right).\\
%%%%%%%%%%%%%%%%
\intertext{If we put $i_1+j_1=r_1$ and $i_2+j_2=r_2$, we obtain}
%%%%%%%%%%%%%%%%%%
=& \sum_{\substack{r_1+r_2=n\\ k_1+k_2=m}}\lambda_1^{r_1}\lambda_2^{r_2}\mu_1^{k_1}\mu_2^{k_2} \\
&\times\left(\sum_{\substack{i_1+j_1=r_1}}\frac{1}{i_1!j_1!k_1!}\int_{0<s<z}L_1(s,z)^{i_1}L_{-1}(s,z)^{j_1}L_0(s,z)^{k_1}\left(\frac{ds}{1-s}+\frac{-ds}{1+s}\right)\right)\\
&\times\left(\sum_{\substack{i_2+j_2=r_2}}\frac{1}{i_2!j_2!k_2!}\int_{0<t<z}L_1(t,z)^{i_2}L_{-1}(t,z)^{j_2}L_0(t,z)^{k_2}\left(\frac{dt}{1-t}+\frac{-dt}{1+t}\right)\right).\\
%%%%%%%%%%%%%%%%%%%%%%
\intertext{By Lemma \ref{lem1} with $s=1$, we obtain}
=&\sum_{\substack{r_1+r_2=n\\ k_1+k_2=m}}\lambda_1^{r_1}\lambda_2^{r_2}\mu_1^{k_1}\mu_2^{k_2}
\left(\sum_{\bm{\alpha}\in S_{r_1,k_1}^{\sgn}} \Li(\bm{\alpha};z)\right)
\left(\sum_{\bm{\beta}\in S_{r_2,k_2}^{\sgn}} \Li(\bm{\beta};z)\right)
\\
=&\sum_{\substack{r_1+r_2=n\\ k_1+k_2=m}}\lambda_1^{r_1}\lambda_2^{r_2}\mu_1^{k_1}\mu_2^{k_2}
\sum_{\substack{\bm{\alpha}\in S_{r_1,k_1}^{\sgn}\\ \bm{\beta}\in S_{r_2,k_2}^{\sgn}}}\Li(\bm{\alpha};z)\Li(\bm{\beta};z).
%&\in\mathbb{Q}[[z]][\lambda_1,\lambda_2,\xi_1,\xi_2,\mu_1,\mu_2].
\end{align*}
Hence, we see that it gives an element of $\mathbb{Q}[[z]][\lambda_1,\lambda_2,\mu_1,\mu_2]$. Thus by operating $\mathfrak{L}_p$ for all primes $p$ and using Lemma \ref{Lem2}, we have
\begin{align}
(\mathfrak{L}_p\;(N^{\sgn}(z))\mathrm{\mod}\;p)_p=& \sum_{\substack{r_1+r_2=n\\ k_1+k_2=m}}\lambda_1^{r_1}\lambda_2^{r_2}\mu_1^{k_1}\mu_2^{k_2}
\sum_{\substack{\bm{\alpha}\in S_{r_1,k_1}^{\sgn}\\ \bm{\beta}\in S_{r_2,k_2}^{\sgn}}}(\mathfrak{L}_p\;(\Li(\bm{\alpha};z)\Li(\bm{\beta};z))\mathrm{\mod}\;p)_p\notag\\
=& \sum_{\substack{r_1+r_2=n\\ k_1+k_2=m}}(-1)^{r_2+k_2+1}\lambda_1^{r_1}\lambda_2^{r_2}\mu_1^{k_1}\mu_2^{k_2}\sum_{\substack{\bm{\alpha}\in S_{r_1,k_1}^{\sgn}\\ \bm{\beta}\in S_{r_2,k_2}^{\sgn}}} \sgn(\bm{\beta})\zf(\bm{\alpha},\bm{\beta}).\label{Nleft}
\end{align}
On the other hand, 
the integral domain of $N^{\sgn}(z)$, denoted by $T$, can be written by the following disjoint union $T=\displaystyle \bigsqcup_{j=1}^3 T_j$ with
\begin{align*}
T_1&:=\{(s,t)\in\mathbb{R}^2\mid 0<s<t<z\}\\
T_2&:=\{(s,t)\in\mathbb{R}^2\mid 0<t<s<z\}\\
T_3&:=\{(s,t)\in\mathbb{R}^2\mid 0<s=t<z\}.
\end{align*} 
We define $N_j^{\sgn}(z)$ by replacing the integral domain of $N^{\sgn}(z)$ with $T_j$ for $1\le j\le 3$.
Since the measure of $T_3$ is equal to $0$, $N_3(z)=0$ holds.
By (\ref{ch}), we have
\begin{align*}
\lambda_1 L_{1}(s,z)+\lambda_2L_{1}(t,z)&=\lambda_1 L_{1}(s,t)+(\lambda_1+\lambda_2)L_{1}(t,z)\\
\lambda_1 L_{-1}(s,z)+\lambda_2L_{-1}(t,z)&=\lambda_1 L_{-1}(s,t)+(\lambda_1+\lambda_2)L_{-1}(t,z)\\
\mu_1 L_{0}(s,z)+\mu_2L_{0}(t,z)&=\mu_1 L_{0}(s,t)+(\mu_1+\mu_2)L_{0}(t,z).
\end{align*}
By the binomial theorem, we obtain 
\begin{align*}
N_1^{\sgn}(z)=&\sum_{\eta_1,\eta_2\in\{\pm 1\}}\int_{0<s<t<z}\sum_{q_1+q_2=n}\frac{1}{q_1!q_2!m!}
(\lambda_1L_1(s,z)+\lambda_2L_1(t,z))^{q_1}
(\lambda_1 L_{-1}(s,z)+\lambda_2 L_{-1}(t,z))^{q_2}\\
&\times (\mu_1 L_{0}(s,z)+\mu_2 L_{0}(t,z))^{m}
\frac{ds}{\eta_1-s}\frac{dt}{\eta_2-t}\\
=&\sum_{\eta_1,\eta_2\in\{\pm 1\}}\sum_{q_1+q_2=n}\frac{1}{q_1!q_2!m!}\int_{0<s<t<z}
(\lambda_1 L_{1}(s,t)+(\lambda_1+\lambda_2)L_{1}(t,z))^{q_1}\\
&\times (\lambda_1 L_{-1}(s,t)+(\lambda_1+\lambda_2)L_{-1}(t,z))^{q_2}
(\mu_1 L_{0}(s,t)+(\mu_1+\mu_2)L_{0}(t,z))^{m}
\frac{ds}{\eta_1-s}\frac{dt}{\eta_2-t}\\
=&\sum_{\substack{r_1+r_2=n \\ k_1+k_2=m}} 
\lambda_1^{r_1}(\lambda_1+\lambda_2)^{r_2}\mu_1^{k_1}(\mu_1+\mu_2)^{k_2}
\int_{0<s<t<z}\sum_{\substack{i_1+j_1=r_1\\ i_2+j_2=r_2}}\sum_{\eta_1,\eta_2\in\{\pm 1\}}\frac{1}{i_1!i_2!j_1!j_2!k_1!k_2!}\\
&\times L_1(s,t)^{i_1}L_1(t,z)^{i_2}L_{-1}(s,t)^{j_1}L_{-1}(t,z)^{j_2}
L_0(s,t)^{k_1}L_0(t,z)^{k_2} \frac{ds}{\eta_1-s}\frac{dt}{\eta_2-t}.\\
%%%%%%%%%%%%%%%%
\intertext{By Lemma \ref{lem1} with $s=2$, we get}
%%%%%%%%%%%%%%%%
=&\sum_{\substack{r_1+r_2=n \\ k_1+k_2=m}} \lambda_1^{r_1}(\lambda_1+\lambda_2)^{r_2}\mu_1^{k_1}(\mu_1+\mu_2)^{k_2}\sum_{\substack{\bm{\alpha}\in S_{r_1,k_1}^{\sgn}\\ \bm{\beta}\in S_{r_2,k_2}^{\sgn}}}\Li(\alpha,\beta).
\end{align*}
By operating $\mathfrak{L}_p$ for all $p$, we have
\begin{align*}
(\mathfrak{L}_p\;(N_1^{\sgn}(z))\;\mathrm{mod}\;p)_p=&\sum_{\substack{r_1+r_2=n \\ k_1+k_2=m}}
\lambda_1^{r_1}(\lambda_1+\lambda_2)^{r_2}\mu_1^{k_1}(\mu_1+\mu_2)^{k_2}\sum_{\substack{\bm{\alpha}\in S_{r_1,k_1}^{\sgn}\\ \bm{\beta}\in S_{r_2,k_2}^{\sgn}}}\zf(\bm{\alpha},\bm{\beta}).
\end{align*}
Similarly, we obtain
\begin{align*}
(\mathfrak{L}_p\;(N_2^{\sgn}(z))\;\mathrm{mod}\;p)_p=&\sum_{\substack{r_1+r_2=n \\ k_1+k_2=m}}
\lambda_2^{r_1}(\lambda_1+\lambda_2)^{r_2}\mu_2^{k_1}(\mu_1+\mu_2)^{k_2} \sum_{\substack{\bm{\alpha}\in S_{r_1,k_1}^{\sgn}\\ \bm{\beta}\in S_{r_2,k_2}^{\sgn}}}\zf(\bm{\alpha},\bm{\beta}).
\end{align*}
Consequently, considering summation of $\mathfrak{L}_p(N_1^{\sgn}(z))$ and $\mathfrak{L}_p(N_2^{\sgn}(z))$, we get
\begin{align}
(\mathfrak{L}_p(N^{\sgn}(z))\;\mathrm{mod}\;p)_p=&\sum_{\substack{r_1+r_2=n \\ k_1+k_2=m}}(\lambda_1+\lambda_2)^{r_2}(\mu_1+\mu_2)^{k_2}(\lambda_1^{r_1}\mu_1^{k_1}+\lambda_2^{r_1}\mu_2^{k_1})\sum_{\substack{\bm{\alpha}\in S_{r_1,k_1}^{\sgn}\\ \bm{\beta}\in S_{r_2,k_2}^{\sgn}}}\zf(\bm{\alpha},\bm{\beta}).\label{Nright}
\end{align}
By (\ref{Nleft}) and (\ref{Nright}), we obtain the claim.
\begin{flushright}
$\square$
\end{flushright}
Theorem \ref{Kama} gives two corollaries by a substitution to (\ref{kamano}):
%Theorem \ref{an}を特殊化することで次の二つの式を得る。
We give also some corollaries derived from Main Theorem.
\begin{cor}\label{cor1}
For $n\ge 2$ and $m \ge 0$, we obtain
\begin{align*}
\sum_{\bm{\alpha}\in S_{n-1,m}^{\sgn}}w_{\sgn}(\bm{\alpha})\zf(\bm{\alpha})
=&(-1)^{m-1}\sum_{\bm{\varepsilon}\in S_{n-2,0}^{\sgn}} \zf(\bm{\varepsilon},m+1)+(-1)^{m}\sum_{\bm{\varepsilon}\in S_{n-2,0}^{\sgn}} \zf(\bm{\varepsilon},\overline{m+1})\\
&-\sum_{k_1+k_2=m}\sum_{\bm{\beta}\in S_{n-2,k_2}^{\sgn}}(\zf(k_1+1,\bm{\beta})+\zf(\overline{k_1+1},\bm{\beta}))
\end{align*}
with
\[
w_{\sgn}(\bm{\alpha}):=
\left\{
\begin{array}{cl}
0&(|\alpha_0|\neq 1)\\
m&(|\alpha_0|=\cdots=|\alpha_{m-1}|=1,\;|\alpha_m|>1)
\end{array}
\right.
\]
for an index $\bm{\alpha}:=(\alpha_0,\ldots,\alpha_n)\in\mathbb{D}^{n+1}$.
\end{cor}
\begin{proof}
By substituting $(\lambda_1,\lambda_2,\mu_1,\mu_2)=(1,0,0,1)$ to (\ref{AA}), we have the above formula.
\end{proof}
\begin{rmk}
Corollary \ref{cor1} is an alternating analogue of \cite{Ka} Corollary 2.3  which says the following:
%\end{rmk}
%\begin{cor}[\cite{Ka}, Corollary 2.3]\label{cork1}
for non negative integers $n\ge 1$, $m\ge 0$,  we obtain
\[
\sum_{\bm{k}\in S_{n-1,q_2}}w(\bm{k})\zf(\bm{k})=(-1)^{m-1}\zf(\underbrace{1,\ldots,1}_{n-1},m+1),
\]
where we define
\[
w(\bm{k}):=
\left\{
\begin{array}{cl}
0&(k_1>1)\\
m&(k_1=\cdots=k_m=1,\; k_{m+1}>1)
\end{array}
\right.
\]
 for $\bm{k}:=(k_1,\ldots,k_n)\in\mathbb{N}^n$.
%\end{cor}
\end{rmk}
\begin{cor}\label{cor2}
For $n \in\mathbb{N}$ and a positive even integer $m$, we obtain
\begin{align*}
\sum_{\bm{\alpha}\in S_{n+1,m}^{\sgn}}J_{\sgn}(\bm{\alpha})\zf(\bm{\alpha})&=
\sum_{\substack{r_1+r_2=n\\ k_1+k_2=m}}(-1)^{r_2-1} 
\sum_{\substack{\alpha\in  S_{r_1,k_1}^{\sgn}\\ \beta\in  S_{r_2,k_2}^{\sgn}}}\sgn(\bm{\beta}) \zf(\bm{\alpha},\bm{\beta})
\end{align*}
with
\[
J_{\sgn}(\bm{\alpha}):=
\left\{
\begin{array}{cc}
(2^{w_{\sgn}(\bm{\alpha})}-1)\sgn (\bm{\alpha})&(w_{\sgn}(\bm{\alpha})\ge 1)\\
0&(w_{\sgn}(\bm{\alpha})=0)
\end{array}
\right.
\]
for an index $\bm{\alpha}:=(\alpha_0,\ldots,\alpha_n)\in\mathbb{D}^{n+1}$.
\end{cor}
\begin{proof}
By substituting $(\lambda_1,\lambda_2,\mu_1,\mu_2)=(1,1,-1,1)$ to (\ref{AA}), we have
\begin{multline*}
\sum_{\substack{r_1+r_2=n\\ k_1+k_2=m}}(-1)^{r_2+m}
\sum_{\substack{\alpha\in  S_{r_1,k_1}^{\sgn}\\ \beta\in  S_{r_2,k_2}^{\sgn}}}\sgn(\bm{\beta}) \zf(\bm{\alpha},\bm{\beta})\\
+\sum_{\substack{r_1+r_2=n}}2^{r_2}((-1)^{m}+1)
\sum_{\substack{\bm{\alpha}\in S_{r_1,m}^{\sgn}\\ \bm{\varepsilon}\in S_{r_2,0}^{\sgn}}}
\zf(\bm{\alpha},\bm{\varepsilon})=(0)_p
\end{multline*}
Since we assume $m$ is even, the second term of above formula is equal to
\begin{align*}
&\sum_{r_1+r_2=n}2^{r_2+1}\sum_{\substack{\bm{\alpha}\in S_{r_1,m}^{\sgn}\\ \bm{\varepsilon}\in S_{r_2,0}^{\sgn}}}\zf(\bm{\alpha},\bm{\varepsilon})
\intertext{By the reversal formula (cf. \cite{zhao} Theorem 4.1)} 
&\zf(s_r,\ldots,s_1)=(-1)^{|\bm{s}|}\sgn (\bm{s}) \zf(s_1,\ldots,s_r),\\
\intertext{for $\bm{s}:=(s_1,\ldots,s_r)\in\mathbb{D}^r$, it can be written as}
=&\sum_{r_1+r_2=n}2^{r_2+1}\sum_{\substack{\bm{\alpha}\in S_{r_1,m}^{\sgn}\\ \bm{\varepsilon}\in S_{r_2,0}^{\sgn}}}\sgn(\bm{\alpha})\sgn(\bm{\varepsilon})\zf(\bm{\varepsilon},\bm{\alpha})\\
=&\sum_{\bm{\alpha}\in S_{n+1,m}^{\sgn}}J_{\sgn}(\bm{\alpha})\zf(\bm{\alpha}).
\end{align*}
\end{proof}
\begin{rmk}
Corollary \ref{cor2} is an alternating analogue of \cite{Ka} Corollary 2.5  which says the following: 
for a non negative integer $q_1$ and a positive even $q_2$, we have
\[
\sum_{\bm{k}\in S_{q_1+1,q_2}}2^{w(\bm{k})}\zf(\bm{k})=(0)_p.
\]
\end{rmk}
\section{General cases}
In this section, we extend the result in our previous section. It is an alternating analog of \cite{Ka} Theorem 3.1.
We follow the notations in \cite{Ka} $\S 3$.
Let take $s$, $t\in\mathbb{N}$ and $W_{s,t}$ be the subset of the $(s+t)$-th symmetric group $\mathfrak{S}_{s+t}$ defined by
\[
W_{s,t}:=\{\sigma\in\mathfrak{S}_{s+t}\mid \sigma(1)<\cdots<\sigma(s),\;\;\sigma(s+1)<\cdots<\sigma(s+t)\}.
\]
Let $\sigma\in W_{s,t}$ and let $\bm{\lambda}:=(\lambda_1,\ldots,\lambda_s)$, $\bm{\lambda'}:=(\lambda_{s+1},\ldots,\lambda_{s+t})$ be tuple of independent.  We define $P_i^{\sigma}(\bm{\lambda},\bm{\lambda}')\in\mathbb{Z}$ ($1\le i\le s+t$) as
\begin{gather*}
\sum_{i=1}^{s-1} \lambda_iL_1(t_i,t_{i+1})+\lambda_sL_1(t_s,z)+\sum_{j=s+1}^{s+t-1} \lambda_j L_1(t_j,t_{j+1})+\lambda_{s+t}L_1(t_{s+t},z)\\
=\sum_{i=1}^{s+t}P_i^{\sigma}(\bm{\lambda},\bm{\lambda'})L_1(t_{\sigma^{-1}(i)},t_{\sigma^{-1}(i+1)}).
\end{gather*}
with $t_{\sigma^{-1}(s+t+1)}=z$. As is observed in \cite{Ka} $\S 3$, $P_i(\bm{\lambda},\bm{\lambda'})$ are uniquely determined.
The following theorem is an alternating variant of \cite{Ka} Theorem 3.1. 
\begin{thm}\label{ann}
Let $n$, $m$, $s$, $t\in\mathbb{N}$. Then,
\begin{align*}
&\sum_{\substack{\bm{r}\in Z_{s+t,n}\\ \bm{k}\in Z_{s+t,m}}}(-1)^{q(\bm{r},\bm{k})}\left(\prod_{l=1}^{s+t}\lambda_l^{r_l}\mu_l^{k_l}\right)
\sum_{(\bm{\alpha}_1,\ldots,\bm{\alpha}_{s+t})\in S_{s+t}}\prod_{r=1}^t \sgn(\bm{\alpha}_{s+r})\zf(\bm{\alpha}_1,\ldots,\bm{\alpha}_s,\bm{\alpha}_{s+t},\ldots,\bm{\alpha}_{s+1})\\
=&\sum_{\sigma\in W_{s,t}}\sum_{\substack{\bm{r}\in Z_{s+t,n}\\ \bm{k}\in Z_{s+t,m}}}\left(\prod_{l=1}^{s+t} P_l^{\sigma}(\bm{\lambda},\bm{\lambda'})^{r_l}P_l^{\sigma}(\bm{\mu},\bm{\mu'})^{k_l}\right)
\sum_{(\bm{\alpha}_1,\ldots,\bm{\alpha}_{s+t})\in S_{s+t}}\zf(\bm{\alpha}_1,\ldots,\bm{\alpha}_{s+t})
\end{align*}
holds 
with $Z_{n,m}:=\{\bm{\bm{\alpha}}\in\mathbb{Z}_{\ge 0}^{n}\mid |\bm{\alpha}|=n+m\}$ for $n$, $m\in\mathbb{N}$ and 
$q(\bm{r},\bm{k}):=r_s+\cdots+r_{s+t}+k_s+\cdots+k_{s+t}+t$ for $\bm{r}:=(r_1,\ldots,r_{s+t})\in Z_{s+t,n}$, $\bm{k}:=(k_1,\ldots,k_{s+t})\in Z_{s+t,m}$.
\end{thm}

\begin{proof}

We define
\begin{align*}
N_{s,t}^{\sgn}(z):=&
\sum_{\substack{\eta_i\in\{\pm1\}\\ 1\le i \le s+t}}\sum_{q_1+q_2=n}\frac{1}{q_1!q_2!m!}\int_{\substack{0<t_1<\cdots<t_s<z\\ 0<t_{s+1}<\cdots<t_{s+t}<z}}\\
&(\lambda_1L_1(t_1,t_2)+\cdots+\lambda_s L_1(t_s,z)+\lambda_{s+1}L_1(t_{s+1},t_{s+2})+\cdots+\lambda_{s+t}L_1(t_{s+t},z))^{q_1}\\
&\times(\lambda_1L_{-1}(t_1,t_2)+\cdots+\lambda_s L_{-1}(t_s,z)+\lambda_{s+1}L_{-1}(t_{s+1},t_{s+2})+\cdots+\lambda_{s+t}L_{-1}(t_{s+t},z))^l\\
&\times(\eta_1L_0(t_1,t_2)+\cdots+\eta_s L_0(t_s,z)+\eta_{s+0}L_0(t_{s+1},t_{s+2})+\cdots+\eta_{s+t}L_0(t_{s+t},z))^m\\
&\times\frac{dt_1}{\eta_1-t_1}\cdots \frac{dt_{s+t}}{\eta_{s+t}-t_{s+t}}.
\end{align*}
By the binomial expansion, we have
\begin{align*}
N_{s,t}^{\sgn}(z)=&
\sum_{\substack{\bm{i}\in Z_{s+t,q_1}\\ \bm{j}\in Z_{s+t,q_2}\\ \bm{k}\in Z_{s+t,q_3}}}\left(\prod_{l=1}^{s+t}\frac{\lambda_l^{i_l+j_l}\mu_l^{k_l}}{i_l!j_l!k_l!}\right)\sum_{\substack{\eta_i\in\{\pm1\}\\ 1\le i \le s+t}}\sum_{q_1+q_2=n}\int_{\substack{0<t_1<\cdots<t_s<z\\ 0<t_{s+1}<\cdots<t_{s+t}<z}}\\
&\begin{gathered}[t]
\times L_1(t_1,t_2)^{i_1}\cdots L_1(t_s,z)^{i_s}L_1(t_{s+1},t_{s+2})^{i_{s+1}}\cdots L_1(t_{s+t},z)^{i_{s+t}}\\
\times L_{-1}(t_1,t_2)^{j_1}\cdots L_{-1}(t_s,z)^{j_s}L_{-1}(t_{s+1},t_{s+2})^{j_{s+1}}\cdots L_{-1}(t_{s+t},z)^{j_{s+t}}\\
\times L_0(t_1,t_2)^{k_1}\cdots L_0(t_s,z)^{k_s}L_0(t_{s+1},t_{s+2})^{k_{s+1}}\cdots L_0(t_{s+t},z)^{k_{s+t}} \frac{dt_1}{\eta_1-t_1}\cdots \frac{dt_{s+t}}{\eta_{s+t}-t_{s+t}}\\
\end{gathered}\\
\intertext{If we put $r_l=i_l+j_l$ for $1\le l \le r+s$,  we have}
=&
\begin{multlined}[t]
\sum_{\substack{\bm{r}\in Z_{s+t,n}\\ \bm{k}\in Z_{s+t,q_3}}}
\left(\prod_{l=1}^{s+t}\frac{\lambda_l^{r_l}\mu_l^{k_l}}{i_l!j_l!k_l!}\right)\sum_{\substack{\eta_i\in\{\pm1\}\\ 1\le i \le s+t}}\sum_{q_1+q_2=n}\\
\times\int_{0<t_1<\cdots<t_s<z}\sum_{\substack{i_l+j_l=r_l\\ 1\le l\le s}}L_1(t_1,t_2)^{i_1}\cdots L_1(t_s,z)^{i_s}L_{-1}(t_1,t_2)^{j_1}\cdots L_{-1}(t_s,z)^{j_s}\\
\end{multlined}\\
&
\begin{multlined}[t]
\times L_0(t_1,t_2)^{k_1}\cdots L_0(t_s,z)^{k_s} \frac{dt_1}{\eta_1-t_1}\cdots \frac{dt_{s}}{\eta_{s}-t_{s}}\\
\times\int_{0<t_{s+1}<\cdots<t_{s+t}<z}\sum_{\substack{i_l+j_l=r_l\\ s+1\le l\le s+t}}L_1(t_{s+1},t_{s+2})^{i_{s+1}}\cdots L_1(t_{s+t},z)^{i_{s+t}}\\
\times L_{-1}(t_{s+1},t_{s+2})^{j_{s+1}}\cdots L_{-1}(t_{s+t},z)^{j_{s+t}}
L_0(t_{s+1},t_{s+2})^{k_{s+1}}\cdots L_0(t_{s+t},z)^{k_{s+t}}\\
\times\frac{dt_{s+1}}{\eta_{s+1}-t_{s+1}}\cdots \frac{dt_{s+t}}{\eta_{s+t}-t_{s+t}}
\end{multlined}\\
\intertext{By using Lemma \ref{lem1}, we obtain}
=&
\sum_{\substack{\bm{r}\in Z_{s+t,n}\\ \bm{k}\in Z_{s+t,q_3}}}
\left(\prod_{l=1}^{s+t}\lambda_l^{r_l}\mu_l^{k_l}\right)\sum_{(\bm{\alpha}_1,\ldots,\bm{\alpha}_{s+t})\in S_{s+t}}\Li_p(\bm{\alpha}_1,\ldots,\bm{\alpha}_s;z)\Li_p(\bm{\alpha}_{s+1},\ldots,\bm{\alpha}_{s+t};z).
\end{align*}
Hence, we see that it gives an element of $\mathbb{Q}[[z]][\lambda_1,\ldots,\lambda_{s+t},\mu_1,\ldots,\mu_{s+t}]$ and by operating $\mathfrak{L}_p$ to above formula for all primes $p$ and using Lemma \ref{Lem2}, we get
\begin{align}
(\mathfrak{L}_p(N_{s,t}^{\sgn}(z)))_p=&
\sum_{\substack{\bm{r}\in Z_{s+t,n}\\ \bm{k}\in Z_{s+t,q_3}}}
(-1)^{q(\bm{r},\bm{k})}
\left(\prod_{l=1}^{s+t}\lambda_l^{r_l}\mu_l^{k_l}\right)
\label{Nstleft}\\
&\times \sum_{(\bm{\alpha}_1,\ldots,\bm{\alpha}_{s+t})\in S_{s+t}}
\prod_{r=1}^t \sgn(\bm{\alpha}_{s+r})\zf(\bm{\alpha}_1,\ldots,\bm{\alpha}_s,\bm{\alpha}_{s+t},\ldots,\bm{\alpha}_{s+1};z).\notag
\end{align}
On the other hand, for $\sigma\in\mathfrak{S}_{s+t}$ and $0<z<1$, we define a domain $D_{\sigma}\subset \mathbb{R}^{s+t}$ by
\[
0<t_{\sigma^{-1}(1)}<\cdots<t_{\sigma^{-1}(s+t)}<z.
\]
We consider the domain, denoted by $D$,
\[
0<t_1<\cdots<t_s<z,\;\;0<t_{s+1}<\cdots<t_{s+t}<z.
\]
We have $\displaystyle \overline{D}=\bigcup_{\sigma\in W_{s,t}} \overline{D_0}$. We note that $\displaystyle \bigcup _{\sigma\in W_{s,t}} D_0$ is disjoint union. Hence,  by the binomial theorem we have
\begin{align*}
N_{s,t}^{\sgn}(z)=&
\begin{multlined}[t]
\sum_{\sigma\in W_{s,t}} \int_{D_\sigma} \sum_{\substack{\eta_i\in\{\pm 1\}\\ 1\le i \le r+s}}\sum_{q_1+q_2=n}\frac{1}{q_1!q_2!m!}(P_1^\sigma(\bm{\lambda},\bm{\lambda'})L_1(t_1,t_2)+\cdots+P_{s+t}^\sigma(\bm{\lambda},\bm{\lambda'})L_1(t_{s+t},z))^{q_1}\\
\times(P_1^\sigma(\bm{\mu},\bm{\mu'})L_{-1}(t_1,t_2)+\cdots+P_{s+t}^\sigma(\bm{\mu},\bm{\mu'})L_{-1}(t_{s+t},z))^{q_2}\\
\times(P_1^\sigma(\bm{\mu},\bm{\mu'})L_0(t_1,t_2)+\cdots+P_{s+t}^\sigma(\bm{\mu},\bm{\mu'})L_0(t_{s+t},z))^m\\
\times\frac{dt_1}{\eta_1-t_1}\cdots \frac{dt_{s+t}}{\eta_{s+t}-t_{s+t}}\\
\end{multlined}
\\
=&\sum_{\sigma\in W_{s,t}}\sum_{\substack{\bm{i}\in Z_{s+t,q_1}\\ \bm{j}\in Z_{s+t,q_2}\\ \bm{k}\in Z_{s+t,q_3}}}
\left(\prod_{l=1}^{s+t} \frac{
P_l^{\sigma}(\bm{\lambda},\bm{\lambda'})^{i_l+j_l}P_l^{\sigma}(\bm{\mu},\bm{\mu'})^{k_l}}{i_l!j_l!k_l!}\right)\\
&\sum_{q_1+q_2=n}\sum_{\substack{\eta_i\in\{\pm 1\}\\ 1\le i \le r+s}}\int_{D_\sigma}L_1(t_1,t_2)^{i_1}\cdots L_1(t_{s+t},z)^{i_{s+t}}L_{-1}(t_1,t_2)^{j_1}\cdots L_{-1}(t_{s+t},z)^{j_{s+t}}\\
&L_0(t_1,t_2)^{k_1}\cdots L_0(t_{s+t},z)^{k_{s+t}}\frac{dt_1}{\eta_1-t_1}\cdots \frac{dt_{s+t}}{\eta_{s+t}-t_{s+t}}.\\
\intertext{If we put $r_l=i_l+j_l$ for $1\le l \le r+s$,  we have}
=&\sum_{\sigma\in W_{s,t}}\sum_{\substack{\bm{r}\in Z_{s+t,n}\\ \bm{k}\in Z_{s+t,m}}}
\left(\prod_{l=1}^{s+t} \frac{
P_l^{\sigma}(\bm{\lambda},\bm{\lambda'})^{i_l+j_l}P_l^{\sigma}(\bm{\mu},\bm{\mu'})^{k_l}}{i_l!j_l!k_l!}\right)\\
&\sum_{q_1+q_2=n}\sum_{\substack{\eta_i\in\{\pm 1\}\\ 1\le i \le r+s}}\sum_{\substack{i_l+j_l=r_l\\ 1\le l\le r+s}}\int_{D_\sigma}L_1(t_1,t_2)^{i_1}\cdots L_1(t_{s+t},z)^{i_{s+t}}L_{-1}(t_1,t_2)^{j_1}\cdots L_{-1}(t_{s+t},z)^{j_{s+t}}\\
&L_0(t_1,t_2)^{k_1}\cdots L_0(t_{s+t},z)^{k_{s+t}}\frac{dt_1}{\eta_1-t_1}\cdots \frac{dt_{s+t}}{\eta_{s+t}-t_{s+t}}.\\
\intertext{By using Lemma \ref{lem1}, we have}
=&\sum_{\sigma\in W_{s,t}}\sum_{\substack{\bm{r}\in Z_{s+t,n}\\ \bm{k}\in Z_{s+t,m}}}
\left(\prod_{l=1}^{s+t} P_l^{\sigma}(\bm{\lambda},\bm{\lambda'})^{i_l+j_l}P_l^{\sigma}(\bm{\mu},\bm{\mu'})^{k_l}\right)\\
&\sum_{(\bm{\alpha}_1,\ldots,\bm{\alpha}_{s+t})\in S_{s+t}}\Li_p(\bm{\alpha}_1,\ldots,\bm{\alpha}_{s+t};z)
\end{align*}
By operating $\mathfrak{L_p}$ to above formula for all primes $p$, we get
\begin{align}
\sum_{\sigma\in W_{s,t}}\sum_{\substack{\bm{r}\in Z_{s+t,n}\\ \bm{k}\in Z_{s+t,m}}}
\left(\prod_{l=1}^{s+t} P_l^{\sigma}(\bm{\lambda},\bm{\lambda'})^{i_l+j_l}P_l^{\sigma}(\bm{\mu},\bm{\mu'})^{k_l}\right)\label{Nstright}\\
&\times \sum_{(\bm{\alpha}_1,\ldots,\bm{\alpha}_{s+t})\in S_{s+t}}\zf(\bm{\alpha}_1,\ldots,\bm{\alpha}_{s+t})\notag
\end{align}
By (\ref{Nstleft}) and (\ref{Nstright}), we obtain the claim.
\end{proof}
\begin{rmk}
Theorem \ref{ann} recovers main theorem in the case $s=t=1$.
\end{rmk}
\appendix
\section{The proof of Lemma \ref{lem1}}\label{Ap}

In this appendix, we give a proof of Lemma \ref{lem1} which is required to prove our main theorem and its generalization (Theorem \ref{ann}).
We prove the lemma by induction on $s$.
Firstly we prove the case for $s=1$.
We define
\begin{align}
I_D^{+,i_1}(z):=&\frac{1}{i_1!j_1!k_1!}\int_D\prod_{l=1}^{i_1}\frac{du_l}{1-u_l}\prod_{m=1}^{j_1}\frac{-dv_m}{1+v_m}\prod_{n=1}^{k_1}\frac{dw_n}{w_n}\frac{dt}{1-t},\label{ap1}\\
I_D^{-,i_1}(z)=&\frac{1}{i_1!j_1!k_1!}\int_D\prod_{l=1}^{i_1}\frac{du_l}{1-u_l}\prod_{m=1}^{j_1}\frac{-dv_m}{1+v_m}\prod_{n=1}^{k_1}\frac{dw_n}{w_n}\frac{-dt}{1+t}\notag
\end{align}
for a suitable domain $D$. We treat an empty product as $1$. We consider the following domains for $0\le i_1\le q_1=i_1+j_1$
\[
D_{i_1}:=\left\{(t_1,u_1,\ldots,u_{i_1},v_1,\ldots,v_{j_1},w_1,\ldots,w_{k_1})\in [0,z)^{i_1+j_1+k_1+1}\middle|
\begin{gathered}
t_1\le u_l,\;\;1\le l \le i_1\\
t_1\le v_m,\;\; 1\le m\le j_1\\
t_1\le w_n,\;\; 1\le n\le k_1
\end{gathered}
\right\}
\]
and the following integration:
\[
I^{\pm}(z):=\sum_{i_1+j_1=q_1}(I_{D_{i_1}}^{+,i_1}(z)+I_{D_{i_1}}^{-,i_1}(z)).
\]
It is clear that
\begin{align*}
&\int_{t_1}^z \frac{dt}{1-t}=\log \frac{1-t_1}{1-z}=L_1(t_1,z),\\
&\int_{t_1}^z \frac{-dt}{1+t}=\log \frac{1+t_1}{1+z}=L_{-1}(t_1,z),\\
&\int_{t_1}^z \frac{dt}{t}=\log \frac{z}{t_1}=L_0(t_1,z).
\end{align*}
The above integrals $I_{D_{i_1}}^{+,i_1}(z)$ and $I_{D_{i_1}}^{-,i_1}(z)$ are written by the following iterated integrals:
\begin{align}
I_{D_{i_1}}^{+,i_1}(z)=\frac{1}{i_1!j_1!k_1!}\int_{0<t_1<z}L_1(t_1,z)^{i_1}L_{-1}(t_1,z)^{j_1}L_0(t_1,z)^{k_1}\frac{dt_1}{1-t_1},\label{1+left}\\
I_{D_{i_1}}^{-,i_1}(z)=\frac{1}{i_1!j_1!k_1!}\int_{0<t_1<z}L_1(t_1,z)^{i_1}L_{-1}(t_1,z)^{j_1}L_0(t_1,z)^{k_1}\frac{-dt_1}{1+t_1}.\label{1-left}
\end{align}
By (\ref{1+left}) and (\ref{1-left}), we obtain
\begin{align}
I^{\pm}(z)=\sum_{i_1+j_1=q_1}\frac{1}{i_1!j_1!k_1!}\int_{0<t_1<z}L_1(t_1,z)^{i_1}L_{-1}(t_1,z)^{j_1}L_0(t_1,z)^{k_1}\left(\frac{dt_1}{1-t_1}+\frac{-dt_1}{1+t_1}\right).\label{1left}
\end{align} 
We put $(x_1,\ldots,x_{q_1+k_1})=(u_1,\ldots,u_{i_1},v_1,\ldots,v_{j_1},w_1,\ldots,w_{k_1})$ and define
\[
D_{\sigma}:=\{(t_1,x_1,\ldots,x_{q_1+k_1})\in D_{i_1} \mid t_1<x_{\sigma(1)}<\cdots<x_{\sigma(q_1+k_1)}\}
\]
for $\sigma\in \mathfrak{S}_{q_1+k_1}$.
Since the closure $\overline{D_{i_1}}$ is given by $\displaystyle \bigcup_{\sigma\in \mathfrak{S}_{q_1+k_1}}\overline{D_{\sigma}}\;$ and $\displaystyle \bigcup_{\sigma\in \mathfrak{S}_{q_1+k_1}}D_{\sigma}$ is disjoint, we have
\begin{align}
I^{\pm}(z)&=\sum_{i_1+j_1=q_1}(I_{D_{i_1}}^+(z)+I_{D_{i_1}}^-(z))=\sum_{i_1+j_1=q_1}\sum_{\sigma\in \mathfrak{S}_{q_1+k_1}}(I_{D_\sigma}^+(z)+I_{D_{\sigma}}^-(z)).\label{form} 
\end{align}
To consider the relationship between $S_{q_1,k_1}^{\sgn}$ and above integrals, we make the following definition:
\begin{align*}
S_{q_1,k_1}^+&=\left\{\bm{\alpha}=(\alpha_0,\ldots,\alpha_{q_1})\in S_{q_1,k_1}^{\sgn}\middle| \prod_{i=0}^{q_1+1}\sgn(\alpha_{i})=1\right\},\\
S_{q_1,k_1}^-&=\left\{\bm{\alpha}=(\alpha_0,\ldots,\alpha_{q_1})\in S_{q_1,k_1}^{\sgn}\middle| \prod_{i=0}^{q_1+1}\sgn(\alpha_{i})=-1\right\},\\
T_{i_1,k_1}^+&=\left\{\bm{\alpha}=(\alpha_0,\ldots,\alpha_{q_1})\in S_{q_1,k_1}^+\middle| \#\left\{p \middle| \prod_{d=0}^{p} \sgn(\alpha_{q_1-d}) =-1\right\}=j_1 \right\},\\
T_{i_1,k_1}^-&=\left\{\bm{\alpha}=(\alpha_0,\ldots,\alpha_{q_1})\in S_{q_1,k_1}^+\middle| \#\left\{p \middle| \prod_{d=0}^{p} \sgn(\alpha_{q_1-d}) =1\right\}=i_1 \right\}.
\end{align*}
We note that $S_{q_1,k_1}^{\sgn}$ is the disjoint union of $S_{q_1,k_1}^+$ and $S_{q_1,k_1}^-$ and that
 $S_{q_1,k_1}^+$ (resp. $S_{q_1,k_1}^-$) is the disjoint union of $T_{i_1,k_1}^+$ (resp. $T_{i_1,k_1}^-$) with $0\le i_1\le q_1$.
For any $i_1$, and $\sigma \in\mathfrak{S}_{i_1+k_1}$, there exists uniquely $\bm{\alpha}\in T_{i_1,k_1}^{+}$ (resp. $T_{i_1,k_1}^{-}$) such that
\begin{align}
I_{D_{\sigma}}^{+,i_1}(z)=\frac{1}{i_1!j_1!k_1!}\Li(\bm{\alpha};z),
\;\;\;
(\text{resp. }I_{D_{\sigma}}^{-,i_1}(z)=\frac{1}{i_1!j_1!k_1!}\Li(\bm{\alpha};z)).
\label{ap2} 
\end{align}
by (\ref{ap1}). %The same things hold for $I_{D_\sigma}^{-,i_1}$.
%It holds as $I_{D_{\sigma}}^{-,i_1}$ respectively.

Conversely, we note that for any $\bm{\alpha}\in T_{i_1,k_1}^{+}$ (resp. $T_{i_1,k_1}^-$), we have $\sigma\in\mathfrak{S}_{q_1+k_1}$ with $i_1!j_1!k_1!$ choices which satisfies (\ref{ap2}).

Hence by ($\ref{form}$),
%the summation of $I_{D_{i_1}}^{+,i_1}$ and $I_{D_{i_1}}^{-,i_1}$ for $0\le i_1\le q_1$ 
we have
\begin{align}
I^{\pm}(z)&= \sum_{i_1+j_1=q_1}\frac{1}{i_1!j_1!k_1!}\sum_{\eta_1\in\{\pm 1\}}\sum_{\sigma\in \mathfrak{S}_{q_1+k_1}}\int_{D_\sigma}\prod_{l=1}^{i_1}\frac{du_l}{1-u_l}\prod_{m=1}^{j_1}\frac{-dv_m}{1+v_m}\prod_{n=1}^{k_1}\frac{dw_n}{w_n}\frac{dt}{\eta_1-t}\label{1right}\\
&= \sum_{i_1+j_1=q_1}\frac{1}{i_1!j_1!k_1!}i_1!j_1!k_1!\left(\sum_{\bm{\alpha}\in T_{i_1,k_1}^+}\Li(\bm{\alpha};z)+\sum_{\bm{\alpha}\in T_{i_1,k_1}^-}\Li(\bm{\alpha};z)\right)\notag\\
%&=\sum_{\bm{\bm{\alpha}}\in S_{q_1,k_1}^+}\Li(\bm{\alpha};z)+\sum_{\bm{\bm{\alpha}}\in S_{q_1,k_1}^-} \Li(\bm{\alpha};z)\notag\\
&=\sum_{\bm{\alpha}\in S_{q_1,k_1}^{\sgn}} \Li(\bm{\alpha};z).\notag
\end{align}
By (\ref{1left}) and (\ref{1right}), we conclude the case for $s=1$.

Secondly, we prove the case for $s>1$.

We define
\begin{align*}
\mathcal{I}:=
\begin{multlined}[t]
\sum_{\substack{\eta_r\in\{\pm1\}\\ 1\le r \le s}}
\sum_{\substack{i_l+j_l=q_l\\ 1\le l\le s}}
C_s^{-1}
\int_{D_{i_1,\ldots,i_s}}
\prod_{l=1}^{i_1+\cdots+i_s}\left(\frac{du_l}{1-u_l}\right)
\prod_{m=1}^{j_1+\cdots+j_s}\left(\frac{-dv_l}{1+v_l}\right)
\prod_{n=1}^{k_1+\cdots+k_s}\left(\frac{dw_l}{w_l}\right)
\\ \times\frac{dt_1}{\eta_1-t_1}\cdots \frac{dt_s}{\eta_s-t_s}
\end{multlined}
\end{align*}
with
\[
D_{i_1,\ldots,i_s}:=\left\{
\begin{gathered}
(t_1,\ldots,t_s)\in[0,z)^{s}\\
(u_1,\ldots,u_{i_1+\cdots+i_s})\in[0,z)^{i_1+\cdots+i_s}\\
(v_1,\ldots,v_{j_1+\cdots+j_s})\in[0,z)^{j_1+\cdots+j_s}\\
(w_1,\ldots,w_{k_1+\cdots+k_s})\in[0,z)^{k_1+\cdots+k_s}
\end{gathered}
\middle|
\begin{gathered}
t_1<\cdots<t_s\\
t_k<u_{i_{k-1}+1},\cdots ,u_{i_k}<t_{k+1} \\
t_k<v_{i_{k-1}+1},\cdots ,v_{i_k}<t_{k+1} \\
t_k<w_{i_{k-1}+1},\cdots ,w_{i_k}<t_{k+1}\\
1\le k\le s,\text{ with $t_{s+1}=z$}
\end{gathered}
\right\}
\]
and $C_s:=\prod_{l=1}^{s} i_l!j_l!k_l!$. By the iterated integral, we have
\begin{align*}
\mathcal{I}=&\sum_{\substack{\eta_r\in\{\pm1\}\\ 1\le r \le s}}
\sum_{\substack{i_l+j_l=q_l\\ 1\le l\le s}}\int_{D_{i_s}}
\frac{1}{i_s!j_s!k_s!}
\left(\prod_{l=1}^{i_s}\frac{du'_l}{1-u'_l}\right)
\left(\prod_{m=1}^{j_s}\frac{-dv'_m}{1+v'_m}\right)
\left(\prod_{n=1}^{k_s}\frac{dw'_n}{w'_n}\right)\frac{dt_{s}}{\eta_{s}-t_{s}}\\
&\times C_{s-1}^{-1}\int_{0<t_1<\cdots<t_{s}}L_1^{i_1}(t_1,t_2)\cdots L_1^{i_{s-1}}(t_{s-1},t_s)
L_{-1}^{j_1}(t_1,t_2)\cdots L_{-1}^{j_{s-1}}(t_{s-1},t_s)\\
&\times L_0^{k_1}(t_1,t_2)\cdots L_0^{k_{s-1}}(t_{s-1},t_s)
\times \frac{dt_1}{\eta_1-t_1}\cdots \frac{dt_{s-1}}{\eta_{s-1}-t_{s-1}}
\end{align*}

with
\[
D_{i_s}:=\left\{(t_s,u'_1\ldots,u'_{i_s},v'_1,\ldots,v'_{j_s},w'_1,\ldots,w'_{k_s})\in[0,z)^{q_s+k_s+1}
\middle|
\begin{gathered}
t_s<u'_{1},\cdots ,u'_{i_s}<z \\
t_s<v'_{1},\cdots ,v'_{j_s}<z \\
t_s<w'_{1},\cdots ,w'_{k_s}<z
\end{gathered}
\right\}.
\]
It is reformulated to be
\begin{align*}
\mathcal{I}=&\sum_{\eta_s\in\{\pm1\}}
\sum_{i_s+j_s=q_s}\frac{1}{i_s!j_s!k_s!}\int_{D_{i_s}}
\left(\prod_{l=1}^{i_s}\frac{du'_l}{1-u'_l}\right)
\left(\prod_{m=1}^{j_s}\frac{-dv'_m}{1+v'_m}\right)
\left(\prod_{n=1}^{k_s}\frac{dw'_n}{w'_n}\right)\frac{dt_{s}}{\eta_{s}-t_{s}}\\
&\times 
\begin{multlined}[t]
\sum_{\substack{\eta_r\in\{\pm1\}\\ 1\le r \le s-1}}\sum_{\substack{i_l+j_l=q_l\\ 1\le l\le s-1}}C(s-1)^{-1}\int_{0<t_1<\cdots<t_{s}}L_1^{i_1}(t_1,t_2)\cdots L_1^{i_{s-1}}(t_{s-1},t_s)\\
\times L_{-1}^{j_1}(t_1,t_2)\cdots L_{-1}^{j_{s-1}}(t_{s-1},t_s)
L_0^{k_1}(t_1,t_2)\cdots L_0^{k_{s-1}}(t_{s-1},t_s)
\frac{dt_1}{\eta_1-t_1}\cdots \frac{dt_{s-1}}{\eta_{s-1}-t_{s-1}}.
\end{multlined}\\
\intertext{By our induction assumption, we obtain}
=&\sum_{\eta_s\in\{\pm1\}}
\sum_{i_s+j_s=q_s}\frac{1}{i_s!j_s!k_s!}\int_{D_{i_s}}
\left(\prod_{l=1}^{i_s}\frac{du'_l}{1-u'_l}\right)
\left(\prod_{m=1}^{j_s}\frac{-dv'_m}{1+v'_m}\right)
\left(\prod_{n=1}^{k_s}\frac{dw'_n}{w'_n}\right)\frac{dt_{s}}{\eta_{s}-t_{s}}\\
&\times 
\sum_{(\bm{\alpha}_1,\ldots,\bm{\alpha}_{s-1})\in S_{s-1}}
\Li(\bm{\alpha}_1,\ldots,\bm{\alpha}_{s-1};t_s).
\intertext{By recursive differential for $\Li(\bm{\alpha};z)$,  we obtain}
&=\sum_{(\bm{\alpha}_1,\ldots,\bm{\alpha}_s)\in S_s}\Li(\bm{\alpha}_1,\ldots,\bm{\alpha}_s;z).
\end{align*}
\begin{flushright}
$\square$
\end{flushright}
\bigskip
\begin{center}
Acknowledgements
\end{center} 

The author is deeply grateful to Professor H. Furusho; without his profound instruction, continuous encouragements, this paper would never be accomplished. He would like to thank S. Kadota and K. Kamano for  giving him comments on the earlier version of the paper.

%He is also grateful to D. Matsuzuki and T. Shinohara who point out a few error of this thesis and they also give advice to him.
%He is also grateful to Kadota and Kamano who gives him some comment about this thesis.
%ｌｅHere, $\sgn(\beta)$ $(\beta\in\mathbb{D})$ is a product of signature of the element of $\beta$.
  
\end{document}